\documentclass[12pt,bezier]{article}
\usepackage{times}
\usepackage{booktabs}
\usepackage{pifont}
\usepackage{floatrow}
\usepackage{caption}
\usepackage{mathrsfs}
\usepackage[fleqn]{amsmath}
\usepackage{amsfonts,amsthm,amssymb,mathrsfs,bbding}
\usepackage{txfonts}
\usepackage{graphics,multicol}
\usepackage{graphicx}
\usepackage{color}
\usepackage{caption}
\usepackage{cite}
\usepackage{latexsym,bm}
\usepackage{indentfirst}
\usepackage{color}
\usepackage[colorlinks=true,anchorcolor=blue,filecolor=blue,linkcolor=blue,urlcolor=blue,citecolor=blue]{hyperref}
\usepackage{extarrows}
\usepackage{cite}
\usepackage{latexsym,bm}
\usepackage{mathtools}
\evensidemargin=\oddsidemargin\topmargin=-1.5cm

\newtheorem{thm}{Theorem}[section]

\newtheorem{ques}{Question}[section]
\newtheorem{lem}{Lemma}[section]

\newtheorem{claim}{Claim}[section]
\theoremstyle{definition}

\addtocounter{section}{0}

\begin{document}
\title{ Signless Laplacian spectral radius of graphs without short cycles or long cycles
\footnote{Supported by the National Natural Science Foundation of China (Nos. 12171066 and 11871222),
Anhui Provincial Natural Science Foundation (Nos. 2108085MA13 and KJ2020B05).}}

\author{{\bf Wenwen Chen}, {\bf Bing Wang}\thanks{Corresponding author. E-mail addresses:
czcwww@chzu.edu.cn (W. Chen); wuyuwuyou@126.com (B. Wang); mqzhai@chzu.edu.cn
(M. Zhai).}, {\bf Mingqing Zhai}\\
{\footnotesize School of Mathematics and Finance, Chuzhou
University, Anhui, Chuzhou, 239012, China}}

\date{}
\maketitle
{\flushleft\large\bf Abstract}
The signless Laplacian spectral radius of a graph $G$, denoted by $q(G)$,
is the largest eigenvalue of its signless Laplacian matrix.
In this paper, we investigate extremal signless Laplacian spectral radius
for graphs without short cycles or long cycles.
Let $\mathcal{G}(m,g)$ be the family of graphs on $m$ edges with girth $g$
 and $\mathcal{H}(m,c)$ be the family of graphs on $m$ edges with circumference $c$.
More precisely, we obtain the unique extremal graph with maximal $q(G)$ in $\mathcal{G}(m,g)$
and $\mathcal{H}(m,c)$, respectively.

\begin{flushleft}
\textbf{Keywords:} Signless Laplacian spectral radius; Extremal graph; Girth; Circumference
\end{flushleft}
\textbf{AMS Classification:} 05C50; 05C35

\section{Introduction}

All graphs considered in this paper are simple, undirected and without isolated vertices.
Let $G$ be a graph with vertex set $V(G)$ and edge set $E(G)$.
The \emph{neighborhood} of a vertex $u\in V(G)$ is denoted by $N_{G}(u)$.
Let $N_G[u]:=N_G(u)\cup \{u\}$, which is called the \emph{closed neighborhood} of $u$.
As usual, $d_G(u)$ is the \emph{degree} of a vertex $u$ and $ \Delta(G)$ is the \emph{maximal degree} of $G$.
The \emph{average 2-degree} of a vertex $u$ is defined as $m_G(u)=\frac{1}{d_G(u)}\sum_{v\in N_G(u)}d_G(v)$.
We use $A(G)$, $D(G)$ and $Q(G)=A(G)+D(G)$ to denote the \emph{adjacency matrix},
\emph{degree diagonal matrix} and \emph{signless Laplacian matrix} of $G$, respectively.
The \emph{spectral radius} $\rho(G)$ and the \emph{signless Laplacian spectral radius} $q(G)$
are the largest moduli of eigenvalues of $A(G)$ and $Q(G)$, respectively.
From Perron-Frobenius theorem, there exists a non-negative unit eigenvector corresponding to $q(G)$,
which is called the \emph{Perron vector} of $Q(G)$.
Moreover, the Perron vector of $Q(G)$ is a positive vector for a connected graph $G$.

A graph $G$ is said to be \emph{$H$-free}, if $G$ does not contain $H$ as a subgraph.
A classic problem in extremal graph theory, known as Tur\'{a}n's problem,
asks what is the maximum number of edges in an $H$-free graph of order $n$?
Nikiforov \cite{Niki11} proposed a spectral version of Tur\'{a}n's problem as follows:
what is the maximum spectral radius of an $H$-free graph of order $n$?
This spectral Tur\'{a}n-type problem attracted much attention in the past decades
(see three surveys \cite{Chen18,Niki11,Li2022survey} and some recent results \cite{CB1,CB2,CLZ,MZ}).
In contrast, the spectral Tur\'{a}n-type problem for graphs with given size can be traced back to Nosal's\cite{Nosal1970} result in 1970,
which states that if $G$ is $C_3$-free then $\rho(G)\leq \sqrt{m}$.
This result was extended by Nikiforov, who proved in \cite{Niki02} that if $G$ is $K_{\omega+1}$-free then $\rho(G)\leq\sqrt{2m(1-1/\omega)}$, and completely characterized the equality in \cite{Niki06}.
In 2007, Bollob\'{a}s and Nikiforov\cite{Bollobas07} posed a stronger conjecture: if $G$ is $K_{\omega+1}$-free then $\lambda_1^2+\lambda_2^2\leq 2m(1-1/\omega) $, where  $\lambda_1$ and $\lambda_2$ are the first two largest eigenvalues of $A(G)$.
Lin, Ning and Wu\cite{Lin21cpc} confirmed Bollob\'{a}s-Nikiforov conjecture for $\omega=2$.
Li, Sun and Yu\cite{Li22} generalized this result by giving an upper bound of $\lambda_1^{2k}+\lambda_2^{2k}$ for $\{C_3,C_5,\ldots,C_{2k+1}\}$-free graphs.
Elphick, Linz and Wocjan\cite{El.Linz.2021+} conjectured that  $\lambda_1^2+\lambda_2^2+\cdots+\lambda_l^2 \leq 2m(1-1/\omega)$ for $K_{\omega+1}$-free graphs, where $l=min(n^+,\omega)$ and $n^+$ is the positive inertia index.

Recently, Gao and Hou \cite{Gao19} characterized the extremal graphs with maximal $\rho(G)$
over all graphs of order $n$ without cycles of length at least $k$.
Very recently, Li, Sun, Yu \cite{Li22} and Lin, Guo \cite{Lin21} independently determined the extremal graphs
with maximal $\rho(G)$ over all non-bipartite graphs of order $n$ without odd cycles of length at most $2k-1$.
In this paper, we consider a variation of above problems by replacing $\rho(G)$ with $q(G)$ and order with size,
that is, what is the maximum $q(G)$ over all graphs of fixed size without short cycles or long cycles?

The \emph{girth} and \emph{circumference} of a graph $G$ are the minimum and maximum lengths of cycles in $G$,
respectively.
We now introduce two families of graphs. For two positive integers $g,c$ with $\min\{g,c\}\geq3$,
let $\mathcal{G}(m,g)$ be the set of graphs on $m$ edges with girth $g$,
and $\mathcal{H}(m,c)$ be the set of graphs on $m$ edges with circumference $c$.
In this paper, we obtain the following two results.

\begin{thm}\label{thgirth}
Let $G_{m,g}$ be the graph obtained from a cycle $C_g$
by linking a vertex of the cycle to $m-g$ isolated vertices.
Then $q(G)\le q(G_{m,g})$ for every $G\in\mathcal{G}(m,g)$,
with equality if and only if $G\cong G_{m,g}$.
\end{thm}

\begin{thm}\label{thcir}
Let $H_{m,c}$ be the graph obtained from a cycle $C_c$ by linking a vertex of the cycle to $c-3$ vertices of $C_c$
and $m-2c+3$ isolated vertices. If $m\geq3c-4$, then $q(H)\le q(H_{m,c})$ for every $H\in\mathcal{H}(m,c)$,
with equality if and only if $H\cong H_{m,c}$.
\end{thm}

\begin{center}\setlength{\unitlength}{0.2mm}\label{fig.1}
\begin{picture}(333.5,213.9)
\qbezier(114.9,83.4)(114.9,60.1)(98.5,43.6)\qbezier(98.5,43.6)(82.0,27.2)(58.7,27.2)\qbezier(58.7,27.2)(35.5,27.2)(19.0,43.6)
\qbezier(19.0,43.6)(2.5,60.1)(2.5,83.4)\qbezier(2.5,83.4)(2.5,106.6)(19.0,123.1)\qbezier(19.0,123.1)(35.5,139.6)(58.7,139.6)
\qbezier(58.7,139.6)(82.0,139.6)(98.5,123.1)\qbezier(98.5,123.1)(114.9,106.6)(114.9,83.4)
\put(58.7,140){\circle*{7}}
\put(28.3,188){\circle*{7}}
\qbezier(58.7,140.7)(43.5,163.9)(28.3,187.1)
\put(0.0,188){\circle*{7}}
\qbezier(58.7,140.7)(29.4,164.2)(0.0,187.8)
\put(117.5,188){\circle*{7}}
\qbezier(58.7,140.7)(88.1,163.9)(117.5,187.1)
\put(55,185){$...$}
\put(95.7,188){\circle*{7}}
\qbezier(58.7,140.7)(77.2,164.6)(95.7,188.5)
\put(14.5,117){\circle*{7}}
\put(103.7,117){\circle*{7}}
\put(7,64){\circle*{7}}
\put(110.9,64){\circle*{7}}
\put(53,42){$...$}
\put(47.9,200){\scriptsize$m-g$}
\put(53,87.7){$C_g$}
\put(327.7,65.3){\circle*{7}}
\qbezier(330.9,84.1)(330.9,60.3)(314.0,43.4)\qbezier(314.0,43.4)(297.2,26.6)(273.3,26.6)\qbezier(273.3,26.6)(249.5,26.6)(232.6,43.4)
\qbezier(232.6,43.4)(215.8,60.3)(215.8,84.1)\qbezier(215.8,84.1)(215.8,107.9)(232.6,124.8)\qbezier(232.6,124.8)(249.5,141.6)(273.3,141.6)
\qbezier(273.3,141.6)(297.2,141.6)(314.0,124.8)\qbezier(314.0,124.8)(330.9,107.9)(330.9,84.1)
\put(273.3,139.9){\circle*{7}}
\put(221.1,188){\circle*{7}}
\qbezier(273.3,139.9)(247.2,163.1)(221.1,188)
\put(248.7,188){\circle*{7}}
\qbezier(273.3,139.9)(261.0,164.2)(248.7,188)
\put(309.6,188){\circle*{7}}
\qbezier(273.3,139.9)(291.5,164.6)(309.6,188)
\put(333.5,188){\circle*{7}}
\qbezier(273.3,139.9)(303.4,164.9)(333.5,188)
\put(222.6,109.5){\circle*{7}}
\put(325,109.5){\circle*{7}}
\put(220.4,65.3){\circle*{7}}
\qbezier(273.3,139.9)(246.9,103.0)(220.4,66.0)
\qbezier(273.3,139.9)(300.5,102.6)(327.7,65.3)
\put(241.4,37){\circle*{7}}
\qbezier(273.3,139.9)(257.4,88.8)(241.4,37.7)
\put(305.2,37.0){\circle*{7}}
\qbezier(273.3,139.9)(289.3,88.5)(305.2,37.0)
\put(269.0,42){$...$}
\put(270,185){$...$}
\put(350,60){$C_c$}
\put(255,200){\scriptsize$m-2c+3$}
\put(90,0){\small Fig. 1: $G_{m,g}$ and $H_{m,c}$.}
\end{picture}
\end{center}

The rest of the paper is organized as follows.
In Section \ref{sec2}, we introduce some tools to study the signless Laplacian spectral radius,
which will be used in subsequent sections.
In Sections \ref{sec3} and \ref{sec4}, we give the proofs of Theorem \ref{thgirth} and Theorem \ref{thcir}, respectively.

\section{Preliminaries}\label{sec2}

The signless Laplacian matrix plays a very important role in spectral graph theory.
In this section, several lemmas on signless Laplacian spectral radius will be introduced.
For more results on signless Laplacian matrix, the readers can refer to three
surveys due to Cvetkovi\'{c} and Simi\'{c} (see \cite{CV1,CV2,CV3}).

The following lemmas concern an operation on edge switching.

\begin{lem}{(Hong and Zhang \cite{Hong05})}\label{leHong05}
Let $G$ be a connected graph, $X$ be a positive eigenvector of $Q(G)$ with $x_i$ corresponding to the vertex $i\in V(G)$,
and $\{v_1,\ldots,v_s\}\subseteq N_G(v)\setminus N_G(u)$
for some two vertices $u,v$ of $G$.
Let $G^*$ be the graph obtained from $G$ by deleting the edges $vv_i$
and adding the edges $uv_i$ for $1\leq i\leq s$.
If $x_{u}\geq x_{v}$, then $q(G^*)>q(G)$.
\end{lem}

The following two lemmas give upper bounds on signless Laplacian spectral radius.

\begin{lem}{(Feng and Yu \cite{Feng09})}\label{le09}
Let $G$ be a connected graph. Then $q(G)\leq \max \{d_G(u)+m_G(u): u\in V(G)\}$,
with equality if and only if $G$ is either a semiregular bipartite graph or a regular graph.
\end{lem}

\begin{lem}{(Zhai, Xue and Lou \cite{Zhai20})}\label{le20}
 Let $G$ be a graph with clique number $\omega$ and size $m$.
 Then $q(G)\leq q(K_\omega^{m-s})$, with equality if and only if $G\cong K_\omega^{m-s}$,
 where $s={\omega\choose 2}$ and $K_\omega^{m-s}$ is obtained from a complete graph $K_\omega$
 by linking $m-s$ edges to a vertex of $K_\omega$.
\end{lem}

Let $k\geq2$. A walk $u_1u_2\ldots u_k$ in a graph $G$ is called an \emph{internal path},
if these $k$ vertices are distinct (except possibly $u_1=u_k$), $\min\{d_G(u_1),d_G(u_k)\}\geq3$ and
$d_G(u_2)=\cdots =d_G(u_{k-1})=2$ (unless $k=2$).
The following lemma concerns an operation on subdividing edges.

\begin{lem}{(Feng, Li and Zhang \cite{LiQF07})}\label{leLiQF07}
Let $G$ be a connected graph and $uv$ be a cut edge on an internal path of $G$.
If we subdivide $uv$, that is, add a new vertex $w$ and substitute $uv$ by a path $uwv$,
and denote the new graph by $G_{uv}$, then $q(G_{uv})<q(G)$.
\end{lem}

Let $Y$ be a real vector.
We denote $Y>\mathbf{0}$, if each coordinate of $Y$ is non-negative and at least one is positive.

\begin{lem}{(Berman and Plemmons \cite{AB})}\label{lem1}
Let $M$ be a non-negative irreducible square matrix with spectral
radius $\lambda(M)$. If there exists a positive vector $Y$ such that
$\alpha Y<MY<\beta Y$, then $\alpha<\lambda(M)<\beta$.
\end{lem}

With the help of Lemmas \ref{leLiQF07} and \ref{lem1}, we obtain the following result by replacing edge subdivision with edge contraction.

\begin{lem}\label{lem2}
Let $G$ be a connected graph
and $uv$ be an edge on an internal path of $G$ with $N_G(u)\cap N_G(v)=\varnothing$.
If we contract $uv$, that is, delete $uv$ and identify $u,v$ as a new vertex $u^*$,
and denote the new graph by $G^{uv}$, then $q(G^{uv})>q(G)$.
\end{lem}

\begin{proof}
If $d_G(u)=2$ or $d_G(v)=2$, then $G$ can be seen as a subdivision of $G^{uv}$,
and the result follows from Lemma \ref{leLiQF07}.
Next, assume that $\min\{d_G(u),d_G(v)\}\geq 3$.

Let $N_G(u)\setminus \{v\}=\{u_1,\ldots,u_s\}$
and $N_G(v)\setminus \{u\}=\{v_1,\ldots,v_t\}$,
where $\min\{s,t\}\geq2$.
To apply Lemma \ref{lem1}, we need to find a positive vector $Y$ such that $Q(G)Y<q(G^{uv})Y$. Let $X$ be the Perron vector of $Q(G^{uv})$, and $Y$ be a  vector defined as
\begin{align*}
 y_w=\left\{
                                       \begin{array}{ll}
                                       \frac 1p\big(\sum_{i=1}^tx_{v_i}+(q-t-1)\sum_{i=1}^sx_{u_i}\big), & \hbox{$w=u$,} \\
                                       \frac 1p\big(\sum_{i=1}^sx_{u_i}+(q-s-1)\sum_{i=1}^tx_{v_i}\big), & \hbox{$w=v$,}\\
                                       x_w, &  \hbox{$w\in V(G)\setminus \{u,v\}$,}
                                       \end{array}
                                     \right.
\end{align*}
where $q=q(G^{uv})$ and $p=(q-t-1)(q-s-1)-1$. Then we have
\begin{align*}
(Q(G)Y)_u&=\sum_{i=1}^sy_{u_i}+y_v+(s+1)y_u\\
&=\sum_{i=1}^sx_{u_i}\!+\!\frac 1p\Big(\sum_{i=1}^sx_{u_i}\!+\!(q\!-\!s\!-\!1)\sum_{i=1}^tx_{v_i}\Big)\!+\!\frac{s+1} {p}\Big(\sum_{i=1}^tx_{v_i}+(q-t-1)\sum_{i=1}^sx_{u_i}\Big)\\
&=qy_u,
\end{align*}
and we can similarly obtain that $(Q(G)Y)_v=qy_v$.

For each vertex $w\in V(G)\setminus(N_G(u)\cup N_G(v))$, we have $y_w=x_w$.
Thus,
\begin{align*}
(Q(G)Y)_w&=\sum_{z\in N_G(w)}y_z+d_G(w)y_w=\sum_{z\in N_{G^{uv}}(w)}x_z+d_{G^{uv}}(w)x_w=qx_w=qy_w.
\end{align*}
Since $X$ is an eigenvector of $G^{uv}$ corresponding to $q(G^{uv})$, we obtain
 \begin{align*}
(q-s-t)x_{u^*}=\sum_{i=1}^sx_{u_i}+\sum_{i=1}^tx_{v_i}.
\end{align*}

Note that $G^{uv}$ contains $K_{1,s+t}$ as a subgraph, we have $q\geq s+t+1$, and hence $Y$ is a positive vector.
Moreover, recall that $\min\{s,t\}\geq2$, it follows that
 \begin{align*}
p&=(q-s-1)(q-t-1)-1\\
&=(q-s-t)(q-t-1)+(t-1)(q-t-1)-1\\
&>(q-s-t)(q-t-1).
\end{align*}
Then we have
 \begin{align*}
y_u-x_{u^*}&=\Big(\frac{q-t-1}p-\frac1{q-s-t}\Big)\sum_{i=1}^sx_{u_i}+\Big(\frac1p-\frac1{q-s-t}\Big)\sum_{i=1}^tx_{v_i}\\
&<\Big(\frac1{(q-s-t)(q-t-1)}-\frac1{q-s-t}\Big)\sum_{i=1}^tx_{v_i}\\
&<0.
\end{align*}
Thus, for each $u_i~(i=1,\ldots,s)$, we have
  \begin{align*}
(Q(G)Y)_{u_i}&=d_{G}(u_i)y_{u_i}+y_u+\sum_{w\in N_G(u_i)\setminus\{u\}}y_w< d_{G^{uv}}(u_i)x_{u_i}+x_{u^*}+\sum_{w\in N_{G^{uv}}(u_i)\setminus\{u\}}x_w=qx_{u_i}=qy_{u_i}.
\end{align*}
By symmetry, $y_v< x_{u^*}$ and $(Q(G)Y)_{v_i}< qy_{v_i}$ for each $v_i~(i=1,\ldots,t)$.

Based on the above analyses, we obtain $Q(G)Y < qY.$
It follows from Lemma \ref{lem1} that $q(G)<q=q(G^{uv}).$
\end{proof}

\section{Proof of Theorem \ref{thgirth}}\label{sec3}

For convenience, we use $|G|$ and $e(G)$ to denote the numbers of vertices and edges of a graph $G$, respectively.
Let $G^*$ denote an extremal graph with maximal signless Laplacian spectral radius in $\mathcal{G}(m,g)$
and $X$ be the Perron vector of $Q(G^*)$ with coordinate $x_v$ corresponding to $v\in V(G)$.
Now we give the proof of Theorem \ref{thgirth}.

\vspace{2mm}
\noindent{\bf Proof of Theorem \ref{thgirth}}

First, we consider the case $g=3$.
By Lemma \ref{le20}, we see that $q(G)\leq q(K_\omega^{m-{\omega\choose 2}})$
for every graph $G$ of size $m$ with clique number $\omega$.
Moreover, if $\omega\geq3$,
then $q(K_\omega^{m-{\omega\choose 2}})$ is strictly decreasing on $\omega$
(see \cite{Zhai20}, Lemma 2.6).
This implies that $K_3^{m-3}$ attains uniquely the maximum signless Laplacian spectral radius
among all graphs of fixed size $m$ with clique number $\omega\geq3$.
Note that $K_3^{m-3}\cong G_{m,3}$ and every graph $G\in\mathcal{G}(m,3)$ has clique number $\omega\geq3$.
It follows that $q(G)\leq q(K_3^{m-3})$ for every $G\in\mathcal{G}(m,3)$,
with equality if and only if $G\cong G_{m,3}$.

\vspace{1mm}
In the following we assume that $g\geq4$. We shall show that $G^*\cong G_{m,g}$.
The proof is divided into five claims.

\begin{claim}\label{cl1}
$G^{*}$ is connected.
\end{claim}

\begin{proof}
Recall that throughout the paper we investigate graphs without isolated vertices.
Suppose that $G^*$ is not connected and it consists of $k$ components $G_1,G_2,\ldots,G_k$.
Then $q(G^*)=q(G_{i_0})$ for some $i_0\in\{1,2,\ldots,k\}$.
Now, select a vertex $u_i\in V(G_i)$ for each $i\in\{1,2,\ldots,k\}$, and
let $G$ be the graph obtained from $G^*$ by identifying $u_1,u_2,\ldots,u_k$.
Then $G\in\mathcal{G}(m,g)$.
Moreover, $G_{i_0}$ is a proper subgraph of $G$, and so $q(G)>q(G_{i_0})=q(G^*).$
This contradicts the choice of $G^*$.
Therefore, $G^{*}$ is connected.
\end{proof}

\begin{claim}\label{cl2}
Let $u_0\in V(G^*)$ with $x_{u_0}=\max_{u\in V(G^*)}x_u$.
If $G^*\ncong C_g$, then $d_{G^*}(u_0)\geq 3$.
\end{claim}

\begin{proof}
Suppose to the contrary that $d_{G^*}(u_0)\leq2$. Then
\begin{align*}
q(G^*)x_{u_0}=d_{G^*}(u_0)x_{u_0}+\sum\limits_{u\in N_{G^*}(u_0)}x_u\leq 4x_{u_0},
\end{align*}
which gives that $q(G^*)\leq 4$.
However, $C_g$ is a proper subgraph of $G^*$, since $G^*\ncong C_g$.
Thus, $q(G^*)>q(C_g)=4,$ a contradiction.
The claim follows.
\end{proof}

\begin{claim}\label{cl3}
There exists a cycle $C$ in $G^*$ with $u_0\in V(C)$.
\end{claim}

\begin{proof}
Let $S$ be the set of vertices which are contained in cycles of $G^*$.
Suppose to the contrary that $u_0\notin S$.
Then we can find a shortest path from $u_0$ to $S$,
say $P:=u_0u_1\ldots u_k$, where $k\geq1$ and $u_k\in S$.
Clearly, $V(P)\cap S=\{u_k\}$, and hence every edge in $E(P)$ is a cut edge of $G^*$.
Now define
\begin{align*}
G=G^*-\{u_ku: u\in N_{G^*}(u_k)\setminus \{u_0\}\}+\{u_0u: u\in N_{G^*}(u_k)\setminus\{u_0\}\}.
\end{align*}
One can observe that $P$ is a pendent path starting from $u_0$ in $G$, and so $G\in \mathcal{G}(m,g)$.
Moreover, since $x_{u_0}\geq x_{u_k}$, we have $q(G)>q(G^*)$ by Lemma \ref{leHong05},
which contradicts the maximality of $q(G^*)$.
Therefore, $u_0\in S$, and the claim holds.
\end{proof}

\begin{claim}\label{cl4}
There exists a cycle $C^*$ of length $g$ in $G^*$ with $u_0\in V(C^*)$.
\end{claim}

\begin{proof}
Let $C$ be a shortest cycle containing $u_0$. We shall show $|C|=g$.
Suppose to the contrary that $|C|\geq g+1$, and let $C=u_0u_1\ldots u_{|C|-1}u_0$.
Since the girth of $G^*$ is $g\geq4$, we have $N_{G^*}(u_0)\cap N_{G^*}(u_1)=\varnothing$.
Now let $G'$ be a graph obtained from $G^*$ by contracting $u_0u_1$ as a new vertex $u^*$ and adding a pendent edge to $u^*$.
Then $e(G')=e(G)=m$, and $q(G')>q(G^*)$ by Lemma \ref{lem2}.
Furthermore, we will see that $G'\in\mathcal{G}(m,g)$.

On the one hand,
since $|C|\geq g+1$, the edge $u_0u_1$ does not belong to any cycle of length $g$ in $G^*$.
Hence, contracting $u_0u_1$ does not destroy cycles of length $g$.
On the other hand, since $C$ is a shortest cycle containing $u_0$ in $G^*$,
$P=u_1\ldots u_{|C|-1}u_0$ is a shortest $(u_0,u_1)$-path in $G^*-\{u_0u_1\}$.
Note that $P$ is of length $|C|-1\geq g$.
Thus, contracting $u_0u_1$ does not give cycles of lengths less than $g$.
Now we obtain that $q(G')>q(G^*)$ and $G'\in G_{m,g}$, which contradicts the choice of $G^*$.
Therefore, the claim holds.
\end{proof}

To complete the proof of Theorem \ref{thgirth}, it suffices to show the following claim.

\begin{claim}
Every edge not on $C^*$ is incident to $u_0$.
\end{claim}

\begin{proof}
Let $E_1$ be the set of edges in $E(G^*)\setminus E(C^*)$ which are not incident to $u_0$.
If $E_1=\varnothing$, then the claim follows.
Now assume that $E_1\neq\varnothing$, and define $E_2=\{u_0w_i: i=1,\ldots,|E_1|\}$,
where $w_1,\ldots,w_{|E_1|}$ are isolated vertices added in $G^*$.
Let $G''=G^*-E_1+E_2$, and
let $X,Y$ be the Perron vectors of $Q(G^*)$ and $Q(G'')$, respectively.
Then
\begin{align}\label{1}
X^{T}Y(q(G'')-q(G^*))=\sum\limits_{u_0w_i\in E_2}(x_{u_0}+x_{w_i})(y_{u_0}+y_{w_i})-\sum\limits_{uv\in E_1}(x_u+x_v)(y_u+y_v).
\end{align}
We now estimate entries in $X$ and $Y$.
Since $w_1,\ldots,w_{|E_1|}$ are isolated vertices in $G^*$ and pendent vertices in $G''$, we have
\begin{align}\label{2}
x_{u_0}+x_{w_{i}}=x_{u_0}  ~~~~\mbox{and} ~~~~y_{u_0}+y_{w_{i}}>y_{u_0}
\end{align}
for each edge $u_0w_i\in E_2$.

Next consider edges in $E_1$.
For each edge $uv\in E_1$, it is obvious that
\begin{align}\label{3}
x_{u}+x_{v}\leq 2x_{u_0}.
\end{align}
Moreover, we will see that if $G^*\ncong K_{2,3}$, then
\begin{align}\label{4}
y_{u}+y_{v}\leq \frac{1}{2}y_{u_0}.
\end{align}

If $u,v\notin V(C^*)\cup N_{G^*}(u_0)$, then $u,v$ are two isolated vertices in $G''$, and so $y_{u}+y_{v}=0\leq\frac{1}{2}y_{u_0}$.
If  $u\in V(C^*)\cup N_{G^*}(u_0)$ and $v\notin V(C^*)\cup N_{G^*}(u_0)$, then $d_{G''}(u)\leq2$ and
$d_{G''}(v)=0$.
Now choose $u^*\in V(C^*)\cup N_{G^*}(u_0)$ such that $y_{u^*}=\max_{w\in (V(C^*)\setminus\{u_0\})\cup N_{G^*}(u_0)}y_w$.
Then $q(G'')y_{u^*}\leq2y_{u^*}+y_{u_0}+y_{u^*}$ and $y_v=0$, which also implies that
$y_{u}+y_{v}\leq y_{u^*}\leq\frac{1}{2}y_{u_0}$ as $q(G'')\geq\Delta(G'')+1\geq5$.
It remains the case $u,v\in V(C^*)\cup N_{G^*}(u_0)$.
Note that $C^*$ is a shortest cycle in $G^*$ and $|C^{*}|\geq4$.
Thus we may assume that $u\in V(C^*)$ and $v\in N_{G^*}(u_0)\setminus V(C^*)$.
Moreover, we can see that the distance between $u$ and $u_0$ in $C^*$ is exactly two.
Now we have $N_{G''}(v)=\{u_0\}$ and so $y_v=\frac{y_{u_0}}{q(G'')-1}$.

Let $u_1,u_{g-1}\in V(C^*)\cap N_{G''}(u_0)$.
By symmetry, $y_{u_{g-1}}=y_{u_1}$,
and clearly, $y_{u_1}>y_w$ for every $w\in N_{G''}(u_0)\setminus\{u_1,u_{g-1}\}$.
We will further see that $y_{u_1}=y_{u^*}$.
Otherwise, $y_{u^*}\neq y_{u_1}$, then $u^*$ is not adjacent to $u_0$.
Thus, $q(G'')y_{u^*}\leq2y_{u^*}+2y_{u^*}$, which gives that $q(G'')\leq4$, a contradiction.
Now choose $u_2\in V(C^*)$ with $y_{u_2}=\max_{w\in V(C^*)\setminus\{u_0,u_1,u_{g-1}\}}y_w$.
Then $q(G'')y_{u_1}\leq2y_{u_1}+y_{u_0}+y_{u_2}$ and
$q(G'')y_{u_2}\leq2y_{u_2}+\sum_{w\in N_{C^*}(u_2)}y_w$.
If $g\geq5$, then $\sum_{w\in N_{C^*}(u_2)}y_w\leq y_{u_1}+y_{u_2}$ and
thus $y_{u_2}\leq \frac{y_{u_1}}{q(G'')-3}\leq\frac12y_{u_1}.$
Combining  $q(G'')y_{u_1}\leq2y_{u_1}+y_{u_0}+y_{u_2}$ gives $y_{u_1}\leq
\frac{y_{u_0}}{q(G'')-\frac52}$ and $y_{u_2}\leq\frac{y_{u_0}}{2q(G'')-5}\leq\frac15y_{u_0}$.
It follows that $y_u+y_v\leq y_{u_2}+\frac14y_{u_0}<\frac12y_{u_0},$ as desired.
If $g=4$, then $m\geq7$ and $q(G'')\geq\Delta(G'')+1\geq6$ as $G^*\ncong K_{2,3}$.
Now $\sum_{w\in N_{C^*}(u_2)}y_w=y_{u_1}+y_{u_{g-1}}=2y_{u_1}$, and hence
$y_{u_2}\leq \frac{2y_{u_1}}{q(G'')-2}\leq\frac12y_{u_1}$.
Combining  $q(G'')y_{u_1}\leq2y_{u_1}+y_{u_0}+y_{u_2}$ gives $y_{u_1}\leq
\frac{y_{u_0}}{q(G'')-\frac52}$ and $y_{u_2}\leq\frac{y_{u_0}}{2q(G'')-5}\leq\frac17y_{u_0}$.
We also have $y_u+y_v\leq y_{u_2}+\frac15y_{u_0}<\frac12y_{u_0}.$

Observe that $|E_1|=|E_2|$. Combining with (\ref{1}-\ref{4}), we obtain that if $G^*\ncong K_{2,3}$, then
\begin{align*}
X^{T}Y(q(G'')-q(G^*))>|E_2|x_{u_0}y_{u_0}-|E_1|x_{u_0}y_{u_0}=0.
\end{align*}
Since $X^{T}Y>0$, we have $q(G'')>q(G^*)$, a contradiction.
If $G^*\cong K_{2,3}$, then $(m,g)=(6,4)$ and $G''\cong G_{6,4}$ (see Fig. \ref{fig.1}).
Straightforward calculation shows that $q(G_{6,4})=3+\sqrt{5}>5=q(K_{2,3})$.
This completes the proof.
\end{proof}

\section{Proof of Theorem \ref{thcir}}\label{sec4}

Recall that $m\geq3c-4$ and $\mathcal{H}(m,c)$ is the set of graphs of size $m$ with circumference $c$.
Note that $\mathcal{H}(m,3)\subseteq\mathcal{G}(m,3)$ and $H_{m,3}\cong G_{m,3}$.
By Theorem \ref{thgirth}, the case $c=3$ is solved.
In the following we assume $c\geq4$.
To prove Theorem \ref{thcir},
we consider a bigger graph family $\mathcal{H}(m,\geq c)$, where $\mathcal{H}(m,\geq c)$
is the set of graphs of size $m$ with circumference at least $c$.
We similarly use $G^*$ to denote an extremal graph with maximal signless Laplacian spectral radius in $\mathcal{H}(m,\geq c)$
and $X$ to denote the Perron vector of $Q(G^*)$ with $x_{u_0}=\max_{u\in V(G^*)}x_u.$
For simplicity, the proof is divided into some claims.

\begin{claim}\label{cl6}
$G^{*}$ is connected.
\end{claim}

\begin{proof}
The proof of connectivity is similar as Claim \ref{cl1}.
\end{proof}

Now denote by $\mathcal{C}_{max}$ the set of longest cycles in $G^*$.
Let $C^*$ have maximal $\sum_{u\in V(C^*)}x_u$ among all cycles in $\mathcal{C}_{max}$.

\begin{claim}\label{cl7}
For each $u\in V(C^*)$ and $v\in V(G^{*})\setminus V(C^*)$, we have $x_{u}\geq x_{v}$, and so $u_0\in V(C^*)$.
\end{claim}

\begin{proof}
Suppose to the contrary that
there exist $u\in V(C^*)$ and $v\in V(G^{*})\setminus V(C^*)$ such that $x_{v}>x_{u}$.
Let $u^{-}$ and $u^{+}$ be the predecessor and the successor of $u$ in $C^*$, respectively.
Since $C^*$ has maximal sum of Perron entries over all longest cycles,
we have $v\notin N_{G^*}(u^{-})\cap N_{G^*}(u^{+})$.
Now we define $G=G^*-\{uu^-,uu^+\}+\{vu^-,vu^+\}$ if $v\notin N_{G^*}(u^{-})\cup N_{G^*}(u^{+})$;
$G=G^*-\{uu^-\}+\{vu^-\}$ if $v\in N_{G^*}(u^+)\setminus N_{G^*}(u^-)$;
and $G=G^*-\{uu^+\}+\{vu^+\}$ if $v\in N_{G^*}(u^-)\setminus N_{G^*}(u^+)$.
Clear, $G\in\mathcal{H}(m,\geq c)$, as $G$ still contains a cycle of length $|C^*|$.
However, by Lemma \ref{leHong05} we have $q(G)>q(G^{*})$, a contradiction.
The claim holds.
\end{proof}

A vertex $u$ in a graph $G$ is called a \emph{dominating vertex}, if $N_G[u]=V(G)$.
If there is a vertex subset $S\subseteq N_G[u]$, then we say that $u$ dominates $S$.

\begin{claim}\label{cl8}
If $uv\in E(G^*)$ with $v\in V(G^{*})\setminus V(C^*)$, then $u$ dominates $V(C^*)$.
\end{claim}

\begin{proof}
Otherwise, say $u'\notin N_{G^*}(u)$ for some $u'\in V(C^*)$,
then by Claim \ref{cl7} $x_{u'}\geq x_v$.
Now we define $G=G^*-\{uv\}+\{uu'\}.$
Then $G\in\mathcal{H}(m,\geq c)$, as $C^*\subseteq G$.
However, by Lemma \ref{leHong05} we have $q(G)>q(G^{*})$, a contradiction.
\end{proof}

A \emph{vertex cover} of a graph $G$ is a vertex subset that covers all edges of $G$.

\begin{claim}\label{cl9}
$V(C^*)$ is a vertex cover of $G^*$.
\end{claim}

\begin{proof}
Suppose that $V(C^*)$ does not cover all edges, that is, there exists an edge $vv'$ with $v,v'\notin V(C^*)$.
Then by Claim \ref{cl8}, both $v$ and $v'$ dominate $V(C^*)$.
Consequently, we can easily find a cycle of length $|C^*|+1$, which contradicts the definition of $C^*$.
\end{proof}

\begin{claim}\label{cl10}
If $V(G^*)\setminus V(C^*)\neq\varnothing$, then $N_{G^*}(v)=\{u_0\}$ for each $v\in V(G^*)\setminus V(C^*)$.
\end{claim}

\begin{proof}
Let $v$ be an arbitrary vertex in $V(G^*)\setminus V(C^*)$.
Then $u_0\in N_{G^*}(v)$
(otherwise, say $u\in N_{G^*}(v)$,
then by Lemma \ref{leHong05} $q(G^*-\{uv\}+\{u_0v\})>q(G^*)$, as $x_{u_0}\geq x_{u}$).
It follows from Claim \ref{cl8} that $u_0$ dominates $V(C^*)$,
and hence $u_0$ dominates $V(G^*)$ by the arbitrariness of $v\in V(G^*)\setminus V(C^*)$.
Next we consider two cases.

\vspace{1.5mm}
\noindent
{\bf Case 1.} $d_{G^{*}}(v)=2$.
\vspace{1.5mm}

Assume that $N_{G^*}(v)=\{u_0,u^*\}$.
Then by Claim \ref{cl9}, $u^*\in V(C^*)$ and so $d_{G^*}(u^*)\geq3$.
Let $G=G^*-\{u^*v\}+\{u_0w\}$, where $w$ is an isolated vertex added in $G^*$.
Clearly, $G\in \mathcal{H}(m,\geq c)$, as $C^*\subseteq G$.
Let $Y$ be the Perron vector of $Q(G)$. Then
\begin{align}\label{5}
q(G)y_{u_0}=(d_{G^*}(u_0)+1)y_{u_0}+\sum_{u\in N_{G^*}(u_0)\setminus \{u^*\}}y_{u}+y_{u^*}+y_w,
\end{align}
\begin{align}\label{6}
q(G)y_{u^*}=(d_{G^*}(u^*)-1)y_{u^*}+\sum_{u\in N_{G^{*}}(u^*)\setminus \{u_0\}}y_{u}+y_{u_0}-y_v.
\end{align}
Note that $N_{G^*}[u^*]\subseteq N_{G^*}[u_0]$ and $N_G(w)=N_G(v)=\{u_0\}$,
then $\sum_{u\in N_{G^{*}}(u^*)}y_u\leq\sum_{u\in N_{G^*}(u_0)}y_u$,
$d_{G^*}(u^*)\leq d_{G^*}(u_0)$ and $y_w=y_v$.
Combining with (\ref{5}-\ref{6}), we have
\begin{align*}
q(G)y_{u_0}-q(G)y_{u^*}\geq (d_{G^{*}}(u^*)+1)y_{u_0}-(d_{G^{*}}(u^*)-1)y_{u^*}+y_{u^*}-y_{u_0}+2y_v.
\end{align*}
It follows that $(q(G)-d_{G^{*}}(u^*))(y_{u_0}+y_v)\geq (q(G)-d_{G^{*}}(u^*)+2)(y_{u^*}+y_v)$.
Equivalently,
\begin{align}\label{7}
y_{u_0}+y_w\geq \frac{q(G)-d_{G^{*}}(u^*)+2}{q(G)-d_{G^{*}}(u^*)}(y_{u^*}+y_{v}),
\end{align}
as $y_v=y_w.$
On the other hand, since $w$ is an isolated vertex in $G^*$, we have
\begin{align}\label{8}
x_{u_0}+x_w=x_{u_0}.
\end{align}
Moreover, $N_{G^*}(v)=\{u_0,u^*\}$ implies that
$q(G^*)x_v=2x_v+x_{u_0}+x_{u^*}\leq 2x_v+2x_{u_0}$, and so
\begin{align}\label{9}
x_{u^*}+x_{v}\leq x_{u_0}+\frac{2}{q(G^*)-2}x_{u_0}=\frac{q(G^*)}{q(G^*)-2}x_{u_0}.
\end{align}
Combining with (\ref{7}-\ref{9}), we obtain that
\begin{align*}
X^TY(q(G)-q(G^*))&=(x_{u_0}+x_w)(y_{u_0}+y_w)-(x_{u^*}+x_v)(y_{u^*}+y_v)\\
&\geq\Big(\frac{q(G)-d_{G^*}(u^*)+2}{q(G)-d_{G^*}(u^*)}-\frac{q(G^*)}{q(G^*)-2}\Big)x_{u_0}(y_{u^*}+y_v)\\
&>0,
\end{align*}
where the last inequality follows from $q(G^*)\geq q(G)$ and $d_{G^*}(u^*)\geq3.$
Therefore, $q(G)>q(G^*)$, a contradiction.

\vspace{1.5mm}
\noindent
{\bf Case 2.} $d_{G^{*}}(v)\geq3$.
\vspace{1.5mm}

We first partition $V(G^*)\setminus V(C^*)$ into $V_1\cup V_2$,
where $V_1$ is the set of pendent vertices.
Clearly, $V_1\subseteq N_{G^*}(u_0);$
and Case 1 implies that $d_{G^*}(v)\geq3$ for each $v\in V_2$.
Now let $K_s^t$ be the graph obtained from $K_s$ by attaching $t$ pendent edges at a vertex of $K_s$.
Then by Lemma \ref{le09}, we have $q(K_s^t)\leq 2(s-1)+t.$
Observe that $G^*\subseteq K_{|C^*|+|V_2|}^{|V_1|}$.
Thus,
\begin{align}\label{10}
q(G^*)\leq q(K_{|C^*|+|V_2|}^{|V_1|})\leq 2(|C^*|+|V_2|-1)+|V_1|.
\end{align}

Now we partition $E(G^*)\setminus E(C^*)$ into $E_1\cup E_2$,
where $E_2$ is the set of chords of $C^*$.
Since by Claim \ref{cl9} $V_1\cup V_2$ is an independent set, we have
$|E_1|=\sum_{v\in V_1\cup V_2}d_{G^*}(v)\geq|V_1|+3|V_2|.$
Moreover, note that $v$ has at least three neighbors.
By Claim \ref{cl9}, $N_{G^*}(v)\subseteq V(C^*)$;
and by Claim \ref{cl8}, each of these neighbors dominates $V(C^*)$.
Thus, $|E_2|\geq 3|C^*|-12.$
Now let $G$ be the graph obtained from $C^*$ by attaching $|E_1|+|E_2|$ pendent edges at $u_0$.
Then, $G\in \mathcal{H}(m,\geq c)$, as $C^*\subseteq G$. Furthermore,
$\Delta(G)=|E_1|+|E_2|+2.$
It follows that
\begin{align}\label{11}
q(G)>q(K_{1,\Delta(G)})=\Delta(G)+1\geq 3|C^*|+|V_1|+3|V_2|-9.
\end{align}
Note that $v\in V_2$ and it has at least three neighbors in $V(C^*)$.
Then $|V_2|\geq1$, and neighbors of $v$ are not consecutive in $C^*$
(otherwise, we have a cycle of length greater than $|C^*|$).
This implies that $|C^*|\geq6$.
Comparing (\ref{10}) with (\ref{11}), we get that
$q(G)>q(K_{|C^*|+|V_2|}^{|V_1|})\geq q(G^*)$, a contradiction.
This completes the proof.
\end{proof}

By Claim \ref{cl10}, $E(G^*)\setminus E(C^*)=E_1\cup E_2$,
where $E_1$ consists of pendent edges incident to $u_0$ and $E_2$ consists of chords of $C^*$.

\begin{claim}\label{cl11}
$|C^*|=c$.
\end{claim}

\begin{proof}
Recall that $|C^*|\geq c\geq 4$.
If $|C^*|=c$, then we are done.
Now suppose that $|C^*|\geq c+1$.
Let $u_1\in V(C^*)$ with $x_{u_1}=\min_{u\in V(C^*)}x_u$.
We will see that $u_1^-u_1^+$ is a chord of $C^*$.
Otherwise, define $G:=G^*-\{u_1u_1^+\}+\{u_1^-u_1^+\}$,
then $G$ has a cycle of length $|C^*|-1\geq c$, and so $G\in\mathcal{H}(m,\geq c)$.
Moreover, since $x_{u_1^-}\geq x_{u_1}$, by Lemma \ref{leHong05} we have $q(G)>q(G^*)$,
a contradiction.

Now $G^*$ contains a $(|C^*|-1)$-cycle $C$ with $V(C)=V(C^*)\setminus\{u_1\}$.
Subsequently, each neighbor of $u_1$ dominates $V(C^*)$,
since $N_{G^*}(u_1)\subseteq V(C^*)$ and $x_u\geq x_{u_1}$ for each $u\in V(C^*)$.
Furthermore, $u_0u_1\in E(G^*)$
(otherwise, $q(G^*-\{u_1^+u_1\}+\{u_0u_1\})>q(G^*)$, as $x_{u_0}\geq x_{u_1^+}$).
It follows that each of $u_1^-$, $u_1^+$ and $u_0$ dominates $V(C^*)$.

Now, if $d_{G^*}(u_1)\geq3$, then similarly as Case 2 of Claim \ref{cl10},
we can get a graph $G\in\mathcal{H}(m,\geq c)$ with $q(G)>q(K_{|C^*|}^{|E_1|})\geq q(G^*)$,
a contradiction.
Therefore, $d_{G^*}(u_1)=2$, which implies that $u_0\in\{u_1^-,u_1^+\}$.
For convenience, we may assume that $C^*=u_0u_1\ldots u_{|C^*|-1}u_0$,
where both $u_0$ and $u_2$ dominate $V(C^*)$.
Next we consider two cases.

\vspace{1.5mm}
\noindent
{\bf Case 1.} There exists a chord of $C^*$ not incident to $u_0$ and $u_2$.
\vspace{1.5mm}

In this case, we can see that $C^*$ has at least $2|C^*|-6$ chords,
and thus $m=e(G^*)\geq |E_1|+3|C^*|-6$.
Let $G$ be the graph obtained from a $(|C^*|-1)$-cycle by attaching
$m-|C^*|+1$ pendent edges at a vertex.
Then $G\in\mathcal{H}(m,\geq c)$ as $|C^*|\geq c+1$,
and $q(G)>q(K_{1,\Delta(G)})=\Delta(G)+1\geq|E_1|+2|C^*|-2.$
On the other hand, since $G^*\subseteq K_{|C^*|}^{|E_1|}$,
we have $q(G^*)\leq q(K_{|C^*|}^{|E_1|})\leq |E_1|+2|C^*|-2.$
It follows that $q(G)>q(G^*)$, a contradiction.

\vspace{1.5mm}
\noindent
{\bf Case 2.} All chords of $C^*$ are incident to $u_0$ or $u_2$.
\vspace{1.5mm}

In this case, we can see that $C^*$ has exactly $2|C^*|-7$ chords,
and so $m=|E_1|+3|C^*|-7$.
Let $G$ be defined as in Case 1.
Then $G\in\mathcal{H}(m,\geq c)$
and $q(G)>q(K_{1,\Delta(G)})=|E_1|+2|C^*|-3.$
On the other hand, note that $d_{G^*}(u_0)=|E_1|+|C^*|-1$, $d_{G^*}(u_1)=2$, $d_{G^*}(u_2)=|C^*|-1$,
$d_{G^*}(u_3)=d_{G^*}(u_{|C^*|})=3$, and $d_{G^*}(u)=4$ for each of other $|C^*|-5$ vertices in $V(C^*)$.
By straightforward computation, we can check that
$d_{G^*}(u)+m_{G^*}(u)\leq |E_1|+2|C^*|-3$ for each $u\in V(G^*)$,
and by Lemma \ref{le09} $q(G^*)\leq |E_1|+2|C^*|-3.$
Therefore, $q(G)>q(G^*)$, a contradiction.
This completes the proof of the claim.
\end{proof}

Now we have Claim \ref{cl10} and Claim \ref{cl11} in hand.
To complete the proof of Theorem \ref{thcir}, it suffices to show the following claim.

\begin{claim}
If $m\geq3c-4$, then $u_0$ dominates $V(C^*)$ and all chords of $C^*$ are incident to $u_0$.
\end{claim}

\begin{proof}
Note that $|C^*|=c$ and $G^*$ contains $|E_1|$ pendent edges.
If $|E_1|=0$, then $q(G^*)\leq q(K_c)=2c-2$.
Let $G$ be the graph obtained from a $c$-cycle by attaching $m-c$ pendent edges at a vertex.
Then $q(G)>\Delta(G)+1=m-c+2\geq2c-2.$
Consequently, $q(G)>q(G^*)$, a contradiction.
Therefore, $|E_1|\geq1$, and
by Claim \ref{cl8} $u_0$ dominates $V(C^*)$.

Recall that $E_2$ is the set of chords of $C^*$.
Let $E_2'$ be the subset of $E_2$ in which each chord is not incident to $u_0$.
In the following it suffices to show $E_2'=\varnothing.$
Suppose to the contrary that $|E_2'|\geq1$.
Note that
\begin{align}\label{12}
|E_2'|+|E_1|=m-2c+3 ~~~\mbox{and} ~~~d_{G^*}(u_0)=|E_1|+c-1.
\end{align}

Now we define $G=G^*-E_2'+\{u_0w_i: i=1,\ldots,|E_2'|\}$,
where $w_1,\ldots,w_{|E_2'|}$ are isolated vertices added in $G^*$.
Then $G\in\mathcal{H}(m,\geq c)$, and by (\ref{12}) we obtain
\begin{align}\label{13}
q(G)>\Delta(G)+1=|E_2'|+d_{G^*}(u_0)+1=|E_2'|+|E_1|+c=m-c+3.
\end{align}

We first assume that $|E_2'|=1$, say $E_2'=\{u_iu_j\}$,
and let $u_k\in V(C^*)$ with $x_{u_k}=\max_{u\in V(G^*)\setminus \{u_0\}}x_u$.
Then
\begin{align*}
q(G^*)x_{u_k}=d_{G^*}(u_k)x_{u_k}+\sum_{u\in N_{G^*}(u_k)}x_u\leq (2d_{G^*}(u_k)-1)x_{u_k}+x_{u_0}.
\end{align*}
It follows that $x_{u_k}\leq \frac{x_{u_0}}{q(G^*)-2d_{G^*}(u_k)+1}$.
If $c=4$, then $G^*\cong K_c^{|E_1|}$ and so $d_{G^*}(u_k)=3$.
Thus $x_{u_i}+x_{u_j}\leq 2x_{u_k}\leq\frac{2}{q(G^*)-5}x_{u_0} $.
If $c\geq 5$, then $d_{G^*}(u_k)\leq4$.
Thus $x_{u_i}+x_{u_j}\leq2x_{u_k}\leq \frac{2}{q(G^*)-7}x_{u_0}$.
Note that $m\geq3c-4$.
By (\ref{13}), we have $q(G^*)\geq q(G)>2c-1$.
Hence, in both cases $x_{u_i}+x_{u_j}<x_{u_0}.$
It follows that
\begin{align*}
q(G)-q(G^*)\geq X^T(Q(G)-Q(G^*))X\geq (x_{u_0}+x_{w_1})^{2}-(x_{u_i}+x_{u_j})^2>0,
\end{align*}
a contradiction.
Therefore, $|E_2'|\geq2$, which also implies that $c\geq5$.

Now let $u^*\in V(G^*)$ such that $d_{G^*}(u^*)+m_{G^*}(u^*)$ is maximal.
If $u^*\notin V(C^*)$,
then $N_{G^*}(u^*)=\{u_0\}$. By Lemma \ref{le09},
$q(G^*)\leq d_{G^*}(u^*)+m_{G^*}(u^*)=1+d_{G^*}(u_0)=|E_1|+c$.
Combining with (\ref{13}), we have $q(G)>q(G^*),$
a contradiction. Hence, $u^*\in V(C^*)$.
In the following, it remains to consider two cases.

\vspace{1.5mm}
\noindent
{\bf Case 1.} $u^*=u$ for some $u\in V(C^*)\setminus \{u_0\}.$
\vspace{1.5mm}

Observe that $d_{G^*}(u)+m_{G^*}(u)\leq d_{G^*}(u_0)+3(d_{G^*}(u)-1)+2|E_2'|$.
Combining with (\ref{12}),
\begin{align*}
d_{G^*}(u)+m_{G^*}(u)\leq 2m-3c+2-|E_1|+3d_{G^*}(u)\leq 2m-3c+1+3d_{G^*}(u),
\end{align*}
as $|E_1|\geq1$.
It follows that
\begin{align*}
q(G^*)\leq d_{G^*}(u)+m_{G^*}(u)\leq d_{G^*}(u)+3+\frac{2m-3c+1}{d_{G^*}(u)}.
\end{align*}
Since $u\in V(C^*)\setminus \{u_0\}$, we have $2\leq d_{G^*}(u)\leq c-1$,
and hence
\begin{align}\label{15}
q(G^*)\leq \max\left\{5+\frac{2m-3c+1}{2},c+2+\frac{2m-3c+1}{c-1}\right\}\leq m-c+3,
\end{align}
as $c\geq5$ and $m\geq3c-4$.
Comparing (\ref{15}) with (\ref{13}), we have $q(G)>q(G^*)$, a contradiction.

\vspace{1.5mm}
\noindent
{\bf Case 2.} $u^*=u_0.$
\vspace{1.5mm}

Since $u_0$ is a dominating vertex, $d_{G^*}(u_0)+m_{G^*}(u_0)=2m-d_{G^*}(u_0).$
It follows that
\begin{align*}
q(G^*)\leq d_{G^*}(u_0)+m_{G^*}(u_0)= d_{G^*}(u_0)-1+\frac{2m}{d_{G^*}(u_0)}.
\end{align*}
Note that $|E_1|\geq1$ and $|E_2'|\geq2$.
Then $d_{G^*}(u_0)=|E_1|+c-1\geq c$.
Moreover, by (\ref{12}) $|E_1|=m-2c+3-|E_2'|\leq m-2c+1$,
and so $d_{G^*}(u_0)=|E_1|+c-1\leq m-c.$
It follows that
\begin{align*}
q(G^*)\leq \max\left\{c-1+\frac{2m}{c}, m-c-1+\frac{2m}{m-c}\right\}\leq m-c+3,
\end{align*}
as $c\geq5$ and $m\geq3c-4$.
Combining with (\ref{13}), we also get a contradiction.
This completes the proof.
\end{proof}

\section{Conclusion remarks}

To end this paper, we present a question for further research.
For $m\geq 3c-4$, Theorem \ref{thcir} determines the maximum $q(G)$ over all graphs in $\mathcal{H}(m,c)$.
A natural question is to consider the case $c+1\leq m\leq3c-5$.

\begin{ques}
For $c+1\leq m\leq3c-5$, what is the maximum signless Laplacian spectral radius over all graphs in $\mathcal{H}(m,c)$?
\end{ques}

\section*{Acknowledgements}
The authors are grateful to the referees for their careful reading and many valuable suggestions.

\end{document}